\documentclass[11pt,reqno,a4paper]{amsart}
\usepackage{amsmath,amssymb,amsfonts} 
\usepackage{graphics}                 
\usepackage{color}       
           \usepackage{amssymb}
\usepackage{amsfonts}
\usepackage{amsmath, amsthm, amssymb, xypic}
\usepackage[english]{babel}
\usepackage[mathcal]{eucal}
\usepackage{fancyhdr}
\usepackage{pdfsync}
\allowdisplaybreaks

\DeclareMathOperator*{\myinf}{in\vphantom{p}f}

\begin{document}

\title{An isometric study of the Lindeberg-Feller CLT via Stein's method}
\author{B. Berckmoes \and R. Lowen \and J. Van Casteren}
\date{}

\thanks{Ben Berckmoes is PhD fellow at the Fund for Scientific Research of Flanders (FWO)}

\maketitle

\newtheorem{pro}{Proposition}[section]
\newtheorem{lem}[pro]{Lemma}
\newtheorem{thm}[pro]{Theorem}
\newtheorem{de}[pro]{Definition}
\newtheorem{co}[pro]{Comment}
\newtheorem{no}[pro]{Notation}
\newtheorem{vb}[pro]{Example}
\newtheorem{vbn}[pro]{Examples}
\newtheorem{gev}[pro]{Corollary}
\newtheorem{vrg}[pro]{Question}

\newtheorem{proA}{Proposition}
\newtheorem{lemA}[proA]{Lemma}
\newtheorem{thmA}[proA]{Theorem}
\newtheorem{deA}[proA]{Definition}
\newtheorem{coA}[proA]{Comment}
\newtheorem{noA}[proA]{Notation}
\newtheorem{vbA}[proA]{Example}
\newtheorem{vbnA}[proA]{Examples}
\newtheorem{gevA}[proA]{Corollary}
\newtheorem{vrgA}[proA]{Question}

\newtheorem{proB}{Proposition}
\newtheorem{lemB}[proB]{Lemma}
\newtheorem{thmB}[proB]{Theorem}
\newtheorem{deB}[proB]{Definition}
\newtheorem{coB}[proB]{Comment}
\newtheorem{noB}[proB]{Notation}
\newtheorem{vbB}[proB]{Example}
\newtheorem{vbnB}[proB]{Examples}
\newtheorem{gevB}[proB]{Corollary}
\newtheorem{vrgB}[proB]{Question}

\newcommand{\Lin}{\textrm{\upshape{Lin}}\left(\left\{\xi_{n,k}\right\}\right)}

\hyphenation{frame-work}
\hyphenation{dif-fe-rent}
\hyphenation{a-vai-la-ble}
\hyphenation{me-tric}
\hyphenation{to-po-lo-gi-cal}
\hyphenation{con-ti--nu-ous-ly}
\hyphenation{de-pen-ding}
\hyphenation{ne-gli-gi-ble}
\hyphenation{de-ri-va-tive}

\begin{abstract}
We use Stein's method to prove a generalization of the Lindeberg-Feller CLT providing an upper and a lower bound for the superior limit of the Kolmogorov distance between a normally distributed random variable and the rowwise sums of a rowwise independent triangular array of random variables which is asymptotically negligible in the sense of Feller. A natural example shows that the upper bound is of optimal order. The lower bound improves a result by Andrew Barbour and Peter Hall. 
\end{abstract}

\section{Introduction and motivation}

\paragraph{}
One of the most important questions in probability theory reads as follows:
\begin{vrg}\label{HeurQ}
Under which conditions is a large sum of independent random variables approximately normally distributed?
\end{vrg}
\paragraph{}
The most common way to address this question formally is within the framework of so-called triangular arrays of random variables, a concept which we now introduce.
\paragraph{}
By a \textit{standard triangular array} (STA)  we mean a triangular array of real square integrable random variables 
\begin{displaymath}
\begin{array}{cccc} 
\xi_{1,1} &  &  \\
\xi_{2,1} & \xi_{2,2} & \\
\xi_{3,1} & \xi_{3,2} & \xi_{3,3} \\
 & \vdots &
\end{array}
 \end{displaymath}
satisfying the following properties.
 \begin{eqnarray*}
 	&(a)& \forall n : \xi_{n,1}, \ldots, \xi_{n,n} \textrm{ are independent.}\\
	&(b)& \forall n, k : \mathbb{E}\left[\xi_{n,k}\right] = 0.\\
	&(c)& \forall n : \sum_{k=1}^{n} \sigma_{n,k}^2 = 1\textrm{, where } \sigma_{n,k}^2 = \mathbb{E}\left[\xi_{n,k}^2\right].
 \end{eqnarray*}
\paragraph{} Let $\left\{\xi_{n,k}\right\}$ be an STA and $\xi$ a standard normally distributed random variable. Question \ref{HeurQ} can now be put as follows:
\begin{vrg}\label{FormQ}
Under which conditions is the weak limit relation $\displaystyle{\sum_{k=1}^n \xi_{n,k} \stackrel{w}{\rightarrow} \xi}$ valid?
\end{vrg}
\paragraph{} It turns out that when $\left\{\xi_{n,k}\right\}$ satisfies \textit{Feller's negligibility condition} in the sense that
\begin{eqnarray}
\max_{k=1}^n \sigma_{n,k}^2 \rightarrow 0, \label{FelNegCond}
\end{eqnarray} 
the Lindeberg-Feller CLT (\cite{Fel}) provides us with a completely satisfactory answer to Question \ref{FormQ}:
\begin{thm}\upshape{(Lindeberg-Feller CLT)}\label{LFCLT}
\textit{Suppose that $\left\{\xi_{n,k}\right\}$ satisfies} \upshape{(\ref{FelNegCond})}. \textit{Then the following are equivalent}:
\begin{eqnarray}
	&(a)&\phantom{1}\sum_{k=1}^n \xi_{n,k} \stackrel{w}{\rightarrow} \xi.\nonumber \\
	&(b)& \forall \epsilon > 0 : \sum_{k=1}^n \mathbb{E}\left[\xi_{n,k}^2 ; \left|\xi_{n,k}\right| \geq \epsilon\right] \rightarrow 0.\label{LinCon}
\end{eqnarray}		
\upshape{(\ref{LinCon})} \textit{is often referred to as} \upshape{Lindeberg's condition}.
\end{thm}

\paragraph{} Throughout the remainder of the paper, $\xi$ and $\{\xi_{n,k}\}$ will be as in Theorem \ref{LFCLT}.

\paragraph{} The Kolmogorov distance between random variables $\eta$ and $\eta^\prime$ is defined by
\begin{displaymath}
K\left(\eta,\eta^\prime\right) = \sup_{x \in \mathbb{R}} \left|\mathbb{P}[\eta \leq x] - \mathbb{P}[\eta^\prime \leq x]\right|.
\end{displaymath}
In general, $K$ is too strong to metrize weak convergence, but it is well known that if $\eta$ is continuously distributed, the following are equivalent for any sequence $\left(\eta_n\right)_n$:
\begin{eqnarray*}
 &(a)& \eta_n \stackrel{w}{\rightarrow} \eta.\\
 &(b)& \limsup_{n \rightarrow \infty} K(\eta,\eta_n) = 0.	
\end{eqnarray*}

\paragraph{} The previous observation reveals that the Lindeberg-Feller CLT in fact gives a necessary and sufficient condition for the number $\displaystyle{\limsup_{n \rightarrow \infty} K\left(\xi,\sum_{k=1}^n \xi_{n,k}\right)}$ to equal $0$. The theorem however does not answer the following question, which is slightly more general than Question \ref{FormQ}, but nevertheless important from both a theoretical and applied point of view:

\begin{vrg}\label{QFormQ}
Under which conditions is the number $\displaystyle{\limsup_{n \rightarrow \infty}K\left(\xi,\sum_{k=1}^n \xi_{n,k}\right)}$ small?
\end{vrg}

\paragraph{} In this paper we will perform what we call an `isometric study' in which we answer Question \ref{QFormQ}, by providing an upper and a lower bound for the number  $\displaystyle{\limsup_{n \rightarrow \infty} K\left(\xi,\sum_{k=1}^n \xi_{n,k}\right)}$, and which constitutes a generalization of the Lindeberg-Feller CLT. Stein's method (\cite{Barbour},\cite{Stein1},\cite{Stein2}) will turn out to be a powerful and indispensable tool for the elaboration of this program.

\section{An upper bound for $\displaystyle{\limsup_{n \rightarrow \infty} K\left(\xi,\sum_{k=1}^n \xi_{n,k}\right)}$}

\paragraph{} We start with Lemma \ref{EasyK}, which makes the task of finding an upper bound for $\displaystyle{\limsup_{n \rightarrow \infty} K\left(\xi,\sum_{k=1}^n \xi_{n,k}\right)}$ considerably more feasible. We let $\mathcal{H}$ stand for the collection of all strictly decreasing functions $h : \mathbb{R} \rightarrow \mathbb{R}$, with a bounded first and second derivative and a bounded and piecewise continuous third derivative, and for which $\displaystyle{\lim_{x \rightarrow -\infty} h(x) = 1}$ and $\displaystyle{\lim_{x \rightarrow \infty} h(x) = 0}$.

\begin{lem}\label{EasyK}
If $\eta$ is continuously distributed, then the formula 
\begin{eqnarray}
\limsup_{n \rightarrow \infty} K\left(\eta,\eta_n\right) = \sup_{h \in \mathcal{H}} \limsup_{n \rightarrow \infty} \left|\mathbb{E}\left[h(\eta) - h(\eta_n)\right]\right|\label{AppForm}
\end{eqnarray}
is valid for any sequence $\left(\eta_n\right)_n$.
\end{lem}

\begin{proof}
Let $\epsilon > 0$ be arbitrary. The continuity of $\eta$ allows us to construct points $x_1 < \cdots < x_N$ such that for each $n$
\begin{eqnarray}
K(\eta,\eta_n) \leq \max_{k=1}^N \left|\mathbb{P}\left[\eta \leq x_k\right] - \mathbb{P}\left[\eta_n \leq x_k\right]\right| + \epsilon.\label{KFP}
\end{eqnarray}
But, again invoking the continuity of $\eta$, it is also easily seen that for each $x \in \mathbb{R}$ there exists $\delta > 0$ and functions $h, \widetilde{h} \in \mathcal{H}$ such that for each $n$
\begin{eqnarray}
\lefteqn{\left|\mathbb{P}\left[\eta \leq x \right] - \mathbb{P}\left[\eta_n \leq x\right]\right|}\nonumber\\ 
&\leq& \max\left\{\mathbb{P}\left[\eta \leq x - \delta\right] - \mathbb{P}\left[\eta_n \leq x\right], \mathbb{P}\left[\eta_n \leq x\right] - \mathbb{P}[\eta \leq x + \delta]\right\} + \epsilon/2\nonumber\\
&\leq& \max \left\{ \mathbb{E}\left[h(\eta) - h(\eta_n)\right] , \mathbb{E}\left[\widetilde{h}(\eta_n) - \widetilde{h}(\eta)\right]\right\} + \epsilon.\label{FPF}
\end{eqnarray}
Combining (\ref{KFP}) and (\ref{FPF}) reveals that there exist functions $h_1, \ldots, h_{2N} \in \mathcal{H}$ such that 
\begin{eqnarray*}
\limsup_{n \rightarrow \infty} K(\eta,\eta_n) &\leq& \limsup_{n \rightarrow \infty} \max_{k=1}^{2N} \left|\mathbb{E}[h_k(\eta) - h_k(\eta_n)]\right| + 2 \epsilon\\
&=& \max_{k=1}^{2N} \limsup_{n \rightarrow \infty}  \left|\mathbb{E}[h_k(\eta) - h_k(\eta_n)]\right| + 2\epsilon\\
&\leq& \sup_{h \in \mathcal{H}} \limsup_{n \rightarrow \infty}  \left|\mathbb{E}[h(\eta) - h(\eta_n)]\right| + 2 \epsilon.
\end{eqnarray*}

Hence the left-hand side of (\ref{AppForm}) is dominated by the right-hand side. 

For the converse inequality, it suffices to remark that for any $h \in \mathcal{H}$
\begin{displaymath}
\left|\mathbb{E}[h(\eta) - h(\eta_n)]\right| \leq \int_0^1 \left|\mathbb{P}\left[\eta \leq h^{-1}t\right] - \mathbb{P}\left[\eta_n \leq h^{-1}t\right]\right| dt \leq K(\eta,\eta_n).
\end{displaymath} 
\end{proof}

\paragraph{} With the Lindeberg-Feller CLT in mind, it seems plausible that bounds for $\displaystyle{\limsup_{n \rightarrow \infty} K\left(\xi,\sum_{k=1}^n \xi_{n,k}\right)}$ should be based on Lindeberg's condition (\ref{LinCon}). Hence, a first naive guess leads  us to the definition of the number
\begin{eqnarray}
\Lin = \sup_{\epsilon > 0} \limsup_{n \rightarrow \infty} \sum_{k=1}^n \mathbb{E}\left[\xi_{n,k}^2;\left|\xi_{n,k}\right| \geq \epsilon\right]\label{LinInd}
\end{eqnarray}
which we call the \textit{Lindeberg index}. It is clear that $\left\{\xi_{n,k}\right\}$ satisfies Lindeberg's condition if and only if $\Lin = 0$. We will provide an example, inspired by one of the problems posed in \cite{Fel}, chapter XV, which illustrates that the Lindeberg index is not trivial in our context.

\paragraph{} Fix $0 < \alpha < 1$, let $\beta =  \frac{\alpha}{1 - \alpha}$\label{def:beta}
and put 
\begin{eqnarray}
s_n^2 = (1 + \beta) n - \beta \sum_{k=1}^n k^{-1} = n + \beta \sum_{k=1}^{n} \left(1 - k^{-1}\right).\label{def:sn}
\end{eqnarray}
Notice that $s_n^2 \rightarrow \infty$. Now consider the STA $\left\{\eta_{\alpha,n,k}\right\}$ such that 
\begin{eqnarray}
\mathbb{P}\left[\eta_{\alpha,n,k} = -1/s_n\right] = \mathbb{P}\left[\eta_{\alpha,n,k} = 1/s_n\right] = \frac{1}{2}\left(1 - \beta k^{-1}\right)\label{def:STAeta1}
\end{eqnarray}
and
\begin{eqnarray}
\mathbb{P}\left[\eta_{\alpha,n,k} = -\sqrt{k}/{s_n}\right] = \mathbb{P}\left[\eta_{\alpha,n,k} = \sqrt{k}/{s_n}\right] = \frac{1}{2}\beta k^{-1}.\label{def:STAeta2}
\end{eqnarray}

\begin{pro}\label{thm:ExplLin}
The STA $\left\{\eta_{\alpha,n,k}\right\}$ satisfies Feller's condition \upshape{(\ref{FelNegCond})} and 
\begin{eqnarray}
\textrm{\upshape{Lin}}\left(\left\{\eta_{\alpha,n,k}\right\}\right) = \alpha.\label{eq:ExplLin}
\end{eqnarray}
\end{pro}

\begin{proof}
For Feller's negligibility condition we observe
\begin{displaymath}
\max_{k=1}^{n} \mathbb{E}\left[\eta_{\alpha,n,k}^2\right] = \frac{(1 + \beta) - \beta n^{-1} }{s_n^2} \rightarrow 0.
\end{displaymath}
\paragraph{} In order to calculate $\textrm{\upshape{Lin}}\left(\left\{\eta_{\alpha,n,k}\right\}\right)$, we fix $\epsilon > 0$ so that $\epsilon^2 (1 + \beta) \leq 1$ and hence $\epsilon^2 s_n^2 \leq n$. Then for $n$ so that $\epsilon s_n > 1$ and $k \leq n$ we have
\begin{displaymath}
\mathbb{E}\left[\eta_{\alpha,n,k}^2 ; \left|\eta_{\alpha,n,k}\right| \geq \epsilon\right] = \left\{\begin{array}{clrr}      
\beta/s_n^2 & \textrm{ if }& k \geq\epsilon^2 s_n^2  \\       
0 & \textrm{ if } & k < \epsilon^2 s_n^2
  \end{array}\right..
\end{displaymath}
It follows that
\begin{eqnarray*}
\lefteqn{\limsup_{n \rightarrow \infty}  \sum_{k=1}^n \mathbb{E}\left[\eta_{\alpha,n,k}^2 ; \left|\eta_{\alpha,n,k}\right| \geq \epsilon \right]}\\
&=& \limsup_{n \rightarrow \infty} \frac{\beta (n - \left\lceil\epsilon^2 s_n^2\right\rceil + 1)}{s_n^2}\\
&=& \limsup_{n \rightarrow \infty} \frac{\beta n}{(1+\beta)n - \beta \sum_{k=1}^n k^{-1}} - \beta \limsup_{n \rightarrow \infty} \frac{\left\lceil \epsilon^2 s_n^2\right\rceil - 1}{s_n^2}\\
&=& \limsup_{n \rightarrow \infty} \frac{\beta}{(1 + \beta) - \beta \frac{1}{n} \sum_{k=1}^n k^{-1}} - \beta \epsilon^2\\
&=&\frac{\beta}{1+\beta} - \beta \epsilon^2\\
&=& \alpha - \beta \epsilon^2
\end{eqnarray*}
which proves the proposition.
\end{proof}

\paragraph{} We now embark on the search for an upper bound for  $\displaystyle{\limsup_{n \rightarrow \infty} K\left(\xi,\sum_{k=1}^n \xi_{n,k}\right)}$ in terms of $\Lin$. We will discuss and compare two different methods, which are both known to lead to a proof of the sufficiency of Lindeberg's condition for normal convergence in the Lindeberg-Feller CLT.

\subsection{The classical method}

\paragraph{} The `classical method' to prove normal convergence appears in many different forms such as e.g. characteristic functions (\cite{Kal}) or Gaussian transforms (\cite{Berg}), but it always involves estimating an expression of the type
\begin{eqnarray}
\left|\mathbb{E}\left[h\left(\xi\right) - h\left(\sum_{k=1}^{n}\xi_{n,k}\right) \right]\right|\label{ExprToEst}
\end{eqnarray}
based on the following three key observations. Their proofs are elementary and can be found in the literature.
\begin{itemize}
 \item[(I)] The random variable $\xi$ is infinitely divisible in the sense that for each $n$
 \begin{eqnarray}
 \xi = \sum_{k=1}^n \eta_{n,k}\label{infdiveq}
 \end{eqnarray}
 where $\{\eta_{n,k}\}$ is the STA consisting of normally distributed random variables with $\mathbb{E}\left[\eta_{n,k}^2\right] = \sigma_{n,k}^2$.
\item[(II)] Let $\{\eta_{n,k}\}$ be as in (I). Then, for any bounded and continuous function $h : \mathbb{R} \rightarrow \mathbb{R}$,
\begin{eqnarray}
\left|\mathbb{E}\left[h\left(\sum_{k=1}^n \eta_{n,k}\right) - h\left(\sum_{k=1}^n \xi_{n,k}\right)\right]\right|
\leq \sum_{k=1}^n \sup_{a \in \mathbb{R}}\left|\mathbb{E}[h\left(a + \eta_{n,k}\right)-h(a + \xi_{n,k})]\right|.\label{stability}
\end{eqnarray}
\item[(III)] Let $h : \mathbb{R} \rightarrow \mathbb{R}$ have a bounded second derivative and a bounded and piecewise continuous third derivative. Then for any $a, x \in \mathbb{R}$
\begin{eqnarray}
\left|h(a + x) - h(a) - h^\prime(a) x  - \frac{1}{2} h^{\prime \prime}(a) x^2\right|
\leq \min \left\{\left\|h^{\prime \prime}\right\|_{\infty} x^2, \frac{1}{6} \left\|h^{\prime \prime \prime}\right\|_\infty \left|x\right|^3\right\}.\label{Taylor}
\end{eqnarray}
\end{itemize}

\paragraph{} Combining (\ref{infdiveq}), (\ref{stability}) and (\ref{Taylor}) for our purpose, yields for any $h \in \mathcal{H}$ and $\epsilon > 0$
\begin{eqnarray*}
\lefteqn{\left|\mathbb{E}\left[h(\xi) - h\left(\sum_{k=1}^n \xi_{n,k}\right)\right]\right|}\\
&=& \left|\mathbb{E}\left[h\left(\sum_{k=1}^n \eta_{n,k}\right) - h\left(\sum_{k=1}^n \xi_{n,k}\right)\right]\right|\\
&\leq& \sum_{k=1}^n \sup_{a \in \mathbb{R}} \left|\mathbb{E}\left[h(a + \eta_{n,k}) - h(a + \xi_{n,k})\right] \right|\\
&\leq& \sum_{k=1}^n \sup_{a \in \mathbb{R}} \mathbb{E}\left[\left|h(a + \eta_{n,k}) - h(a) -  h^\prime(a) \eta_{n,k} - \frac{1}{2} h^{\prime \prime}(a) \eta_{n,k}^2\right|\right]\\
&&+\sum_{k=1}^n \sup_{a \in \mathbb{R}} \mathbb{E}\left[\left|h(a + \xi_{n,k}) - h(a) - h^\prime(a) \xi_{n,k} - \frac{1}{2} h^{\prime \prime}(a) \xi_{n,k}^2\right|\right]\\
&\leq& \frac{1}{6} \left\|h^{\prime \prime \prime}\right\|_\infty \sum_{k=1}^n \mathbb{E}\left[\left|\eta_{n,k}\right|^3\right]  + \frac{1}{6} \left\|h^{\prime \prime \prime}\right\|_{\infty} \sum_{k=1}^n \mathbb{E}\left[\left|\xi_{n,k}\right|^3;\left|\xi_{n,k}\right| < \epsilon\right]\\
&&+ \left\|h^{\prime \prime}\right\|_\infty \sum_{k=1}^n \mathbb{E}\left[\xi_{n,k}^2 ; \left|\xi_{n,k}\right| \geq \epsilon\right] \\
&\leq& \frac{1}{6} \left\|h^{\prime \prime \prime}\right\|_\infty \left( \mathbb{E}\left[\left|\xi\right|^3\right] \max_{k=1}^{n} \sigma_{n,k} + \epsilon \right) + \left\|h^{\prime \prime}\right\|_\infty \sum_{k=1}^n \mathbb{E}\left[\xi_{n,k}^2 ; \left|\xi_{n,k}\right| \geq \epsilon\right]. 
\end{eqnarray*}
Because of (\ref{FelNegCond}), the previous calculation shows that for any $h \in \mathcal{H}$
\begin{eqnarray}
\limsup_{n \rightarrow \infty} \left|\mathbb{E}\left[h(\xi) - h\left(\sum_{k=1}^n \xi_{n,k}\right)\right]\right| \leq \left\|h^{\prime \prime}\right\|_\infty \Lin\label{ClassResult}.
\end{eqnarray}
\paragraph{} Recalling (\ref{AppForm}), we see that (\ref{ClassResult}) proves that Lindeberg's condition is sufficient for the weak limit relation $\displaystyle{\sum_{k=1}^n \xi_{n,k} \stackrel{w}{\rightarrow} \xi}$ to hold. However, since $\left\|h^{\prime \prime}\right\|_\infty$ blows up if we let $h$ run through the collection $\mathcal{H}$, (\ref{ClassResult}) is useless for the derivation of an upper bound for $\displaystyle{\limsup_{n \rightarrow \infty} K\left(\xi,\sum_{k=1}^n \xi_{n,k}\right)}$. We conclude that although the classical method suffices to decide when the number $\displaystyle{\limsup_{n \rightarrow \infty} K\left(\xi,\sum_{k=1}^n \xi_{n,k}\right)}$ is $0$, it is not subtle enough to decide when it is small.

\subsection{Stein's method}

\paragraph{} Whereas the classical method provides an upper bound for (\ref{ExprToEst}) based on a direct analysis of the function $h$, Stein's method first transforms (\ref{ExprToEst}) to an expression which is easier to analyze. The basics of the method are contained in the following lemma. The proofs can be found in e.g. \cite{Barbour}.

\begin{lem}\label{SteinBasics}
Let $h : \mathbb{R} \rightarrow \mathbb{R}$ be measurable and bounded. Put
\begin{eqnarray}
f_h(x) = e^{x^2/2} \int_{-\infty}^x \left(h(t) - \mathbb{E}[h(\xi)]\right) e^{-t^2/2} dt \label{SteinSolution}.
\end{eqnarray}
Then for any $x \in \mathbb{R}$
\begin{eqnarray}
\mathbb{E}\left[h(\xi)\right] - h(x)= x f_h(x) - f_h^\prime (x).\label{SteinIdentity}
\end{eqnarray}
Moreover, if $h$ is absolutely continuous, then
\begin{eqnarray}
\left\|f_h^{\prime \prime}\right\|_\infty \leq 2\left\|h^\prime\right\|_\infty\label{fhbounded},
\end{eqnarray}
and if $h_z = 1_{\left]-\infty,z\right]}$ for $z \in \mathbb{R}$, then for all $x,y \in \mathbb{R}$
\begin{eqnarray}
\left|f^\prime_{h_z}(x) - f^\prime_{h_z} (y)\right| \leq 1.\label{SteinUpperBound}
\end{eqnarray}
\end{lem}

\paragraph{} We will try to apply Lemma \ref{SteinBasics} in order to derive an upper bound for $\displaystyle{\limsup_{n \rightarrow \infty} K\left(\xi,\sum_{k=1}^n \xi_{n,k}\right)}$. We first need three additional lemmata.

\paragraph{} Stein's method was used by Barbour and Hall to derive Berry-Esseen type bounds in \cite{BHall}. The following lemma is inspired by their paper.

\begin{lem}
Let $h \in \mathcal{H}$ and put 
\begin{eqnarray}
\delta_{n,k} = f_h\left(\sum_{i \neq k} \xi_{n,i} + \xi_{n,k}\right) - f_{h}\left(\sum_{i \neq k} \xi_{n,i}\right) - \xi_{n,k} f^\prime_h\left(\sum_{i \neq k} \xi_{n,i}\right)\label{defdelta}
\end{eqnarray}
and 
\begin{eqnarray}
\epsilon_{n,k} = f^\prime_h\left(\sum_{i \neq k} \xi_{n,i} + \xi_{n,k}\right) - f^\prime_{h}\left(\sum_{i \neq k} \xi_{n,i}\right) - \xi_{n,k} f^{\prime \prime}_h\left(\sum_{i \neq k}\xi_{n,i}\right).\label{defepsilon}
\end{eqnarray}
Then
\begin{eqnarray}
\lefteqn{\mathbb{E}\left[\left(\sum_{k=1}^n \xi_{n,k}\right) f_h\left(\sum_{k=1}^n \xi_{n,k}\right) - f_h^\prime\left(\sum_{k=1}^n \xi_{n,k}\right)\right]}\nonumber\\
&& =\sum_{k=1}^n \mathbb{E}\left[\xi_{n,k}\delta_{n,k}\right] - \sum_{k=1}^n \sigma_{n,k}^2 \mathbb{E}\left[\epsilon_{n,k}\right].\label{BHallIneq}
\end{eqnarray}
\end{lem}

\begin{proof}
Calculate the right-hand side of (\ref{BHallIneq}) and recall that $\xi_{n,k}$ and $\sum_{i \neq k} \xi_{n,i}$ are independent, $\mathbb{E}\left[\xi_{n,k}\right] = 0$ and $\sum_{k=1}^n \sigma_{n,k}^2 = 1$.
\end{proof}

\paragraph{} The following lemma is a straightforward application of Taylor's theorem.

\begin{lem}
Let $f : \mathbb{R} \rightarrow \mathbb{R}$ have a bounded derivative and a bounded and piecewise continuous second derivative. Then for any $a, x \in \mathbb{R}$
\begin{eqnarray}
\lefteqn{\left|f(a + x) - f(a) - f^\prime(a) x \right|}\nonumber\\ 
&&\leq \min \left\{\left(\sup_{x_1,x_2 \in \mathbb{R}}\left|f^\prime(x_1) - f^{\prime}(x_2)\right|\right) \left|x\right|,\frac{1}{2} \left\|f^{\prime \prime}\right\|_\infty x^2\right\}.\label{Taylorf}
\end{eqnarray}
\end{lem}

\paragraph{} We finally observe that the favorable inequality (\ref{SteinUpperBound}) is extendable to all functions in the collection $\mathcal{H}$:

\begin{lem}
Let $h \in \mathcal{H}$. Then for all $x,y \in \mathbb{R}$
\begin{eqnarray}
\left|f^\prime_h(x) - f^\prime_h(y)\right| \leq 1.\label{DiffFh}
\end{eqnarray}
\end{lem} 

\begin{proof}
From (\ref{SteinSolution}) we derive that
\begin{eqnarray}
f_h^\prime(x) = x e^{x^2/2} \int_{-\infty}^x \left(h(t) - \mathbb{E}[h(\xi)]\right) e^{-t^2/2} dt + h(x) - \mathbb{E}[h(\xi)].\label{fhRep}
\end{eqnarray}
Furthermore, for all $h \in \mathcal{H}$ we have
\begin{eqnarray}
h(x) = \int_0^1 h_{h^{-1} s}(x) ds.\label{hRep}
\end{eqnarray}
Combining (\ref{fhRep}) and (\ref{hRep}) and applying Fubini yields 
\begin{eqnarray}
f^\prime_h(x) - f^\prime_h(y) = \int_0^1 \left[f^\prime_{h^{-1}s}(x) - f^\prime_{h^{-1}s}(y)\right] ds\label{DiffRep} 
\end{eqnarray}
and the lemma follows from (\ref{SteinUpperBound}).
\end{proof}

\paragraph{} Combining (\ref{SteinIdentity}), (\ref{fhbounded}), (\ref{BHallIneq}), (\ref{Taylorf}) and (\ref{DiffFh}) yields for $h \in \mathcal{H}$ and $\epsilon > 0$
\begin{eqnarray*}
\lefteqn{\left|\mathbb{E}\left[h\left(\xi\right) - h\left(\sum_{k=1}^{n}\xi_{n,k}\right) \right]\right|}\\
 &=& \left|\mathbb{E}\left[\left(\sum_{k=1}^n \xi_{n,k}\right) f_h\left(\sum_{k=1}^n \xi_{n,k}\right) - f_h^\prime\left(\sum_{k=1}^n \xi_{n,k}\right)\right]\right|\\
&\leq&  \sum_{k=1}^n \mathbb{E}\left[\left|\xi_{n,k}\delta_{n,k}\right|\right] + \sum_{k=1}^n \sigma_{n,k}^2 \mathbb{E}\left[\left|\epsilon_{n,k}\right|\right]\\
&=& \frac{1}{2}\left\|f_h^{\prime \prime}\right\|_\infty \sum_{k=1}^n \mathbb{E}\left[\left|\xi_{n,k}\right|^3 ; \left|\xi_{n,k}\right|<\epsilon\right]\\
&& + \left(\sup_{x_1,x_2 \in \mathbb{R}} \left|f_h^\prime(x_1) - f_{h}^\prime(x_2)\right|\right) \sum_{k=1}^n \mathbb{E}\left[\left|\xi_{n,k}\right|^2;\left|\xi_{n,k}\right|\geq \epsilon\right]\\
&& + \left(\sup_{x_1, x_2 \in \mathbb{R}} \left|f_h^{\prime \prime}(x_1) - f_h^{\prime \prime} (x_2)\right| \right)\sum_{k=1}^n\sigma_{n,k}^2 \mathbb{E}\left[\left|\xi_{n,k}\right|\right]\\
&\leq& \frac{1}{2} \left\|f_h^{\prime \prime}\right\|_\infty \epsilon +  \sum_{k=1}^n \mathbb{E}\left[\left|\xi_{n,k}\right|^2;\left|\xi_{n,k}\right|\geq \epsilon\right] \\
&&+ \left(\sup_{x_1,x_2 \in \mathbb{R}} \left|f_h^{\prime \prime}(x_1) - f_h^{\prime \prime}(x_2)\right|\right) \max_{k=1}^n \sigma_{n,k}.
\end{eqnarray*}

\paragraph{} Because of (\ref{FelNegCond}), the previous calculation shows that for any $h \in \mathcal{H}$
\begin{eqnarray}
\limsup_{n \rightarrow \infty} \left|\mathbb{E}\left[h(\xi) - h\left(\sum_{k=1}^n \xi_{n,k}\right)\right]\right| \leq \Lin\label{SteinResult}.
\end{eqnarray}
\paragraph{} Recalling (\ref{AppForm}), we see that (\ref{SteinResult}) entails the following beautiful generalization of the sufficiency of Lindeberg's condition in the Lindeberg-Feller CLT.

\begin{thm}\label{MainThm}
The inequality 
\begin{eqnarray}
\limsup_{n \rightarrow \infty} K\left(\xi,\sum_{k=1}^n \xi_{n,k}\right) \leq \Lin\label{ineqMainThm}
\end{eqnarray}
is valid.
\end{thm}

\section{A lower bound for $\displaystyle{\limsup_{n \rightarrow \infty} K\left(\xi,\sum_{k=1}^n \xi_{n,k}\right)}$}

\paragraph{} In \cite{BHall} the following theorem is proved.

\begin{thm}\label{BHallLowerBoundTheorem}
There exists a constant $C > 0$, not depending on $\left\{\xi_{n,k}\right\}$, such that 
\begin{eqnarray}
\limsup_{n \rightarrow \infty} \sum_{k=1}^n \mathbb{E}\left[\xi_{n,k}^2 \left(1 - e^{-\frac{1}{4} \xi_{n,k}^2}\right)\right] \leq C \limsup_{n \rightarrow \infty} K\left(\xi,\sum_{k=1}^n \xi_{n,k}\right).\label{BHallIneq2}
\end{eqnarray}
Moreover, (\ref{BHallIneq2}) can be shown to hold with $C \leq 41$.
\end{thm}

\paragraph{} Let $\Phi_L$ be the collection of all non-decreasing functions $\phi : \mathbb{R}^+ \rightarrow \left[0,1\right]$  which are strictly positive on $\mathbb{R}^+_0$ and for which $\displaystyle{\lim_{\epsilon \downarrow 0} \phi(\epsilon) = 0}$.
\paragraph{} Inspired by Theorem \ref{BHallLowerBoundTheorem}, we define for $\phi \in \Phi_L$ the number
\begin{eqnarray}
\widetilde{\textrm{\upshape{Lin}}}_\phi\left(\left\{\xi_{n,k}\right\}\right) = \limsup_{n \rightarrow \infty} \sum_{k=1}^n \mathbb{E}\left[\xi_{n,k}^2 \phi(\left|\xi_{n,k}\right|)\right]
\end{eqnarray}
which we call the \textit{relaxed Lindeberg index (with respect to $\phi$)}.  Furthermore, we put for $\gamma > 0$
\begin{eqnarray}
 \phi_{\gamma}(x) = 1 - e^{-\gamma x^2}
\end{eqnarray}
and
\begin{eqnarray}
\widetilde{\textrm{\upshape{Lin}}}_{\gamma}\left(\left\{\xi_{n,k}\right\}\right) = \widetilde{\textrm{\upshape{Lin}}}_{\phi_\gamma}\left(\left\{\xi_{n,k}\right\}\right).
\end{eqnarray}
\paragraph{} The following proposition collects some basic properties of the relaxed Lindeberg index.
\begin{pro}\label{thm:PropRelLin}
For any $\phi \in \Phi_L$ we have
\begin{eqnarray}
\widetilde{\textrm{\upshape{Lin}}}_\phi\left(\left\{\xi_{n,k}\right\}\right) \leq \Lin\label{RelLinLeqLin}
\end{eqnarray}
and, in addition, 
\begin{eqnarray}
\left\{\xi_{n,k}\right\} \textrm{ satisfies Lindeberg's condition }  \Leftrightarrow \widetilde{\textrm{\upshape{Lin}}}_\phi \left(\left\{\xi_{n,k}\right\}\right) = 0.\label{LCByRelLin}
\end{eqnarray}
Furthermore, for any sequence $(\phi_n)_n$ in $\Phi_L$ we have
\begin{eqnarray}
\phi_n \uparrow 1_{]0, \infty[} \Rightarrow \widetilde{\textrm{\upshape{Lin}}}_{\phi_n}\left(\left\{\xi_{n,k}\right\}\right) \uparrow \Lin.\label{gensupLin}
\end{eqnarray}
In particular,
\begin{eqnarray}
 \gamma \uparrow \infty \Rightarrow \widetilde{\textrm{\upshape{Lin}}}_\gamma\left(\left\{\xi_{n,k}\right\}\right) \uparrow \Lin.\label{supLin}
\end{eqnarray}
Finally, for any $\gamma > 0$,
\begin{eqnarray}
\widetilde{\textrm{\upshape{Lin}}}_{2 \gamma}\left(\left\{\xi_{n,k}\right\}\right) \leq 2 \widetilde{\textrm{\upshape{Lin}}}_{\gamma}\left(\left\{\xi_{n,k}\right\}\right).\label{CompLin}
\end{eqnarray} 
\end{pro}

\begin{proof}
(\ref{RelLinLeqLin}) follows from
\begin{eqnarray*}
\lefteqn{\sum_{k=1}^n \mathbb{E}\left[\xi_{n,k}^2 \phi\left(\left|\xi_{n,k}\right|\right)\right]}\\
&=& \sum_{k=1}^n \mathbb{E}\left[\xi_{n,k}^2 \phi\left(\left|\xi_{n,k}\right|\right);\left|\xi_{n,k}\right| \geq \epsilon \right] + \sum_{k=1}^n \mathbb{E}\left[\xi_{n,k}^2 \phi\left(\left|\xi_{n,k}\right|\right) ; \left|\xi_{n,k}\right| < \epsilon\right]\\
&\leq& \sum_{k=1}^n \mathbb{E}\left[\xi_{n,k}^2 ; \left|\xi_{n,k}\right| \geq \epsilon \right] + \phi(\epsilon)
\end{eqnarray*}
and the fact that $\displaystyle{\lim_{\epsilon \downarrow 0} \phi(\epsilon) = 0}$.
\paragraph{} Notice that (\ref{RelLinLeqLin}) entails that $\widetilde{\textrm{\upshape{Lin}}}_\phi \left(\left\{\xi_{n,k}\right\}\right) = 0$ if $\left\{\xi_{n,k}\right\}$  satisfies Lindeberg's condition. For the converse implication, observe that, for any $\epsilon > 0$,
\begin{displaymath}
\phi(\epsilon) \sum_{k=1}^n \mathbb{E}\left[\xi_{n,k}^2 ; \left|\xi_{n,k}\right| > \epsilon\right] \leq \sum_{k=1}^n \mathbb{E}\left[\xi_{n,k}^2 \phi\left(\left|\xi_{n,k}\right|\right)\right]
\end{displaymath}
and $\phi(\epsilon) > 0$. This proves (\ref{LCByRelLin}).
\paragraph{} In order to prove (\ref{gensupLin}), we choose for $\epsilon > 0$ a natural number $n_0$ such that 
\begin{eqnarray*}
\phi_{n_0}(\epsilon) \geq 1 - \epsilon.
\end{eqnarray*}
Then
\begin{displaymath}
\limsup_{n \rightarrow \infty} \sum_{k=1}^n \mathbb{E}\left[\xi_{n,k}^2 ; \left|\xi_{n,k}\right| > \epsilon\right] \leq \limsup_{n \rightarrow \infty} \sum_{k=1}^n \mathbb{E}\left[\xi_{n,k}^2 \phi_{n_0}\left(\left|\xi_{n,k}\right|\right)\right] + \epsilon
\end{displaymath}
and  (\ref{gensupLin}) follows.
\paragraph{} Finally, (\ref{CompLin}) follows from the observation that
\begin{eqnarray*}
\lefteqn{\sum_{k=1}^{n} \mathbb{E}\left[\xi_{n,k}^2 \phi_{2 \gamma}\left(\left|\xi_{n,k}\right|\right)\right]}\\
&=&\sum_{k=1}^n \mathbb{E}\left[\xi_{n,k}^2 \phi_\gamma\left(\left|\xi_{n,k}\right|\right)\right] +  \sum_{k=1}^n \mathbb{E}\left[\xi_{n,k}^2 \phi_{\gamma}\left(\left|\xi_{n,k}\right|\right) e^{-\gamma \xi_{n,k}^2}\right].
\end{eqnarray*}
\end{proof}

\paragraph{} Proposition \ref{thm:ExplRelLin} provides an explicit formula for $\widetilde{\textrm{\upshape{Lin}}}_{\gamma}\left(\left\{\eta_{\alpha,n,k}\right\}\right)$, where $\{\eta_{\alpha,n,k}\}$ is the STA determined by (\ref{def:sn}), (\ref{def:STAeta1}) and (\ref{def:STAeta2}). Recall that $\{\eta_{\alpha,n,k}\}$ satisfies Feller's condition (\ref{FelNegCond}) and that
$\textrm{\upshape{Lin}}\left(\left\{\eta_{\alpha,n,k}\right\}\right) = \alpha.$ (Proposition \ref{thm:ExplLin})

\begin{pro}\label{thm:ExplRelLin}
The formula
\begin{eqnarray}
\widetilde{\textrm{\upshape{Lin}}}_{\gamma}\left(\left\{\eta_{\alpha,n,k}\right\}\right) = \alpha \left(1 - \frac{1 - e^{-\gamma(1-\alpha)}}{\gamma (1 - \alpha)}\right)\label{eq:ExplRelLin}
\end{eqnarray}
is valid.
\end{pro}

\begin {proof}
An easy calculation leads to 
\begin{eqnarray*}
\lefteqn{\sum_{k=1}^n \mathbb{E}\left[\eta_{\alpha,n,k}^2 \phi_\gamma\left(\left|\eta_{\alpha,n,k}\right|\right)\right]}\\
&=& \frac{n}{s_n^2} \left(1 - e^{-\frac{\gamma}{s_n^2}}\right) - \beta e^{-\frac{\gamma}{s_n^2}} \left(\frac{1}{s_n^2} \sum_{k=1}^n k^{-1}\right) + \beta \frac{n}{s_n^2} - \beta \left(\frac{1}{s_n^2} \sum_{k=1}^n e^{-\frac{\gamma k}{s_n^2}}\right).
\end{eqnarray*}
Arguing analogously as in the proof of Proposition \ref{thm:ExplLin}, we see that 
\begin{displaymath}
\frac{n}{s_n^2} \rightarrow \frac{1}{1 + \beta} \textrm{ and }  \frac{1}{s_n^2} \sum_{k=1}^n k^{-1} \rightarrow 0.
\end{displaymath}
Also, 
\begin{displaymath}
\frac{1}{s_n^2} \sum_{k=1}^n e^{-\frac{\gamma k}{s_n^2}} = e^{-\frac{\gamma}{s_n^2}} \frac{1 - e^{-\frac{n\gamma}{s_n^2}}}{s_n^2\left(1 - e^{-\frac{\gamma}{s_n^2}}\right)} \rightarrow \frac{1 - e^{- \frac{\gamma}{1 + \beta}}}{\gamma}.
\end{displaymath}
Recalling that $\alpha = \frac{\beta}{1 + \beta}$, the previous limit relations yield the desired result.
\end{proof}

\paragraph{} In terms of the relaxed Lindeberg index, Theorem \ref{BHallLowerBoundTheorem} is turned into

\begin{thm}\label{BHallLowerBoundTheoremNew}
There exists a constant $C_{1/4} > 0$, not depending on $\left\{\xi_{n,k}\right\}$, such that 
\begin{eqnarray}
\widetilde{\textrm{\upshape{Lin}}}_{1/4}\left(\left\{\xi_{n,k}\right\}\right) \leq C_{1/4} \limsup_{n \rightarrow \infty} K\left(\xi,\sum_{k=1}^n \xi_{n,k}\right).\label{BHallIneq2new}
\end{eqnarray}
Moreover, (\ref{BHallIneq2new}) can be shown to hold with $C_{1/4} \leq 41$.
\end{thm}

\paragraph{} Theorem \ref{BHallLowerBoundTheoremNew} generalizes the necessity of Lindeberg's condition in the Lindeberg-Feller CLT and therefore constitutes a counterpart for Theorem \ref{MainThm}. It also has the following

\begin{gev}
For each $\gamma > 0$ there exists a constant $C_\gamma > 0$, not depending on $\left\{\xi_{n,k}\right\}$, such that 
\begin{eqnarray}
\widetilde{\textrm{\upshape{Lin}}}_{\gamma}\left(\left\{\xi_{n,k}\right\}\right) \leq C_\gamma \limsup_{n \rightarrow \infty} K\left(\xi,\sum_{k=1}^n \xi_{n,k}\right).
\end{eqnarray}
\end{gev}

\begin{proof}
Combine (\ref{CompLin}) and (\ref{BHallIneq2new}).
\end{proof}

\paragraph{} In the same spirit we have obtained Theorem \ref{thmBHallIneqour}, which is stronger than Theorem \ref{BHallLowerBoundTheoremNew}. It provides the sharpest lower bound for $\displaystyle{\limsup_{n \rightarrow \infty} K\left(\xi,\sum_{k=1}^n \xi_{n,k}\right)}$ we have produced so far. We defer the proof to Appendix A.

\begin{thm}\label{thmBHallIneqour}
There exists a constant $C_{1/2} > 0$, not depending on $\left\{\xi_{n,k}\right\}$, such that 
\begin{eqnarray}
\widetilde{\textrm{\upshape{Lin}}}_{1/2}\left(\left\{\xi_{n,k}\right\}\right) \leq C_{1/2} \limsup_{n \rightarrow \infty} K\left(\xi,\sum_{k=1}^n \xi_{n,k}\right).\label{BHallIneq2our}
\end{eqnarray}
Moreover, (\ref{BHallIneq2our}) can be shown to hold with $C_{1/2} \leq 30.3$.
\end{thm}

\section{Conclusion}

\paragraph{} Theorem \ref{MainThm} and Theorem \ref{thmBHallIneqour} together yield Theorem \ref{MainResultLowUp}, our main result, which answers Question \ref{QFormQ} by providing an upper and a lower bound for the number  $\displaystyle{\limsup_{n \rightarrow \infty} K\left(\xi,\sum_{k=1}^n \xi_{n,k}\right)}$ which constitute a generalization of the Lindeberg-Feller CLT.

\begin{thm}\label{MainResultLowUp}
There exists a constant $\widetilde{C}_{1/2} > 0$, not depending on $\left\{\xi_{n,k}\right\}$, such that 
\begin{eqnarray}
\widetilde{C}_{1/2} \widetilde{\textrm{\upshape{Lin}}}_{1/2}\left(\left\{\xi_{n,k}\right\}\right) \leq \limsup_{n \rightarrow \infty} K\left(\xi,\sum_{k=1}^n \xi_{n,k}\right) \leq \Lin.\label{MainResultIneqs}
\end{eqnarray}
Moreover, we can take $\widetilde{C}_{1/2} \geq 0.033$.
\end{thm}

\paragraph{} Applying (\ref{eq:ExplLin}), (\ref{eq:ExplRelLin}) and (\ref{MainResultIneqs}) to the STA $\{\eta_{\alpha,n,k}\}$, determined by (\ref{def:sn}), (\ref{def:STAeta1}) and (\ref{def:STAeta2}), we obtain

\begin{thm}\label{thm:ExplMainIneqs}
There exists a constant $\widetilde{C}_{1/2} > 0$ such that for each $\alpha > 0$ 
\begin{eqnarray}
\widetilde{C}_{1/2} \alpha \left(1 - \frac{1 - e^{-\frac{1}{2}(1-\alpha)}}{\frac{1}{2} (1 - \alpha)}\right) \leq \limsup_{n \rightarrow \infty} K\left(\xi,\sum_{k=1}^n \eta_{\alpha,n,k}\right) \leq \alpha.\label{ExplMainIneqs}
\end{eqnarray}
Moreover, we can take $\widetilde{C}_{1/2} \geq 0.033$.
\end{thm}

\paragraph{} From the lower bound for $\displaystyle{\limsup_{n \rightarrow \infty} K\left(\xi,\sum_{k=1}^n \eta_{\alpha,n,k}\right)}$ in (\ref{ExplMainIneqs}) we conclude that the upper bound in (\ref{MainResultIneqs}) is of optimal order in the sense that

\begin{thm} 
There does not exist a constant $C>0$, not depending on $\left\{\xi_{n,k}\right\}$, such that 
\begin{eqnarray}
\limsup_{n \rightarrow \infty} K\left(\xi,\sum_{k=1}^n \xi_{n,k}\right) \leq C \left[\textrm{\upshape{Lin}}\left(\left\{\xi_{n,k}\right\}\right)\right]^{1 + p}
\end{eqnarray}
with $p > 0$.
\end{thm}

\paragraph{} We end this paper with the following open question. 

\begin{vrg}\label{LBQ}
Does there exist a constant $C > 0$, not depending on $\left\{\xi_{n,k}\right\}$, such that 
\begin{eqnarray}
\Lin \leq C \limsup_{n \rightarrow \infty} K\left(\xi,\sum_{k=1}^n \xi_{n,k}\right)?
\end{eqnarray}
\end{vrg}

\paragraph{}  A positive answer to Question \ref{LBQ} is equivalent with the existence of a constant $C > 0$, not depending on $\left\{\xi_{n,k}\right\}$,  such that for each $\gamma > 0$
\begin{eqnarray*}
\widetilde{\textrm{\upshape{Lin}}}_{\gamma}\left(\left\{\xi_{n,k}\right\}\right) \leq C \limsup_{n \rightarrow \infty} K\left(\xi,\sum_{k=1}^n \xi_{n,k}\right).
\end{eqnarray*}

\section*{Appendix A : Proof of Theorem \ref{thmBHallIneqour} }

\paragraph{} We assume w.l.o.g. that $\xi$ and $\left\{\xi_{n,k}\right\}$ are independent. Furthermore, we let $\left\{\widetilde{\xi}_{n,k}\right\}$ be a copy of $\left\{\xi_{n,k}\right\}$ not depending on $\xi$ and $\left\{\xi_{n,k}\right\}$.
\\
\paragraph{} Put
\begin{eqnarray}
\psi_{1/2}(x) = 1 - \int_0^1 e^{-\frac{1}{2}(1-s^2)x^2} ds.
\end{eqnarray}
The following lemma reveals that $\psi_{1/2}$ is closely related to $\phi_{1/2}$.

\begin{lemA}
The inequalities
\begin{eqnarray}
\psi_{1/2} \leq \phi_{1/2} \leq \frac{3}{2}\psi_{1/2}\label{PreCompPsi}
\end{eqnarray}
are valid.
\end{lemA}

\begin{proof}
The first inequality is obvious. For the second inequality, we observe that
\begin{eqnarray*}
\phi_{1/2}(x) - \psi_{1/2}(x) &=& e^{-\frac{1}{2}x^2} \int_0^1 \left(e^{\frac{1}{2} s^2 x^2} - e^0\right) ds\\
&=& \frac{1}{2} e^{-\frac{1}{2} x^2} \int_0^1 \int_0^1 s^2 x^2  e^{\frac{1}{2} t s^2 x^2} dt ds\\
&\leq& \frac{1}{2} e^{-\frac{1}{2} x^2} \int_0^1 s^2 x^2 e^{\frac{1}{2} s^2 x^2} ds\\
&=& \frac{1}{2} e^{-\frac{1}{2} x^2} \int_0^1 s de^{\frac{1}{2}s^2 x^2}\\
&=& \frac{1}{2} e^{-\frac{1}{2} x^2} \left(e^{\frac{1}{2}x^2} - \int_0^1 e^{\frac{1}{2} s^2 x^2} ds\right)\\
&=& \frac{1}{2} \psi_{1/2}(x)
\end{eqnarray*}
and we are done.
\end{proof}

\paragraph{} It follows from (\ref{PreCompPsi}) that $\psi_{1/2}$ belongs to $\Phi_L$. Also,

\begin{gevA}
\begin{eqnarray}
\widetilde{\textrm{\upshape{Lin}}}_{\psi_{1/2}}\left(\left\{\xi_{n,k}\right\}\right) \leq \widetilde{\textrm{\upshape{Lin}}}_{1/2}\left(\left\{\xi_{n,k}\right\}\right) \leq \frac{3}{2}\widetilde{\textrm{\upshape{Lin}}}_{\psi_{1/2}}\left(\left\{\xi_{n,k}\right\}\right).\label{CompPsi}
\end{eqnarray}
\end{gevA}

\begin{lemA}
Let $f : \mathbb{R} \rightarrow \mathbb{R}$ be bounded, measurable and antisymmetric and put 
\begin{eqnarray}
g(x) = xf(x). 
\end{eqnarray}
Then, for $a \in \mathbb{R}$,
\begin{eqnarray}
\mathbb{E}\left[a^2 \int_0^1 g\left(\xi + sa\right) e^{-\frac{1}{2}\left(1-s^2\right)a^2} ds\right] = \mathbb{E}\left[a f(\xi + a)\right].\label{basicRelation}
\end{eqnarray}
\end{lemA}

\begin{proof}
First suppose that $f$ is continuously differentiable. Then
\begin{eqnarray*}
\lefteqn{\mathbb{E}\left[a^2 \int_0^1 g(\xi + sa) e^{-\frac{1}{2} (1-s^2) a^2} ds\right]}\\
&=& \mathbb{E}\left[a^2 \int_0^1 \xi f(\xi + sa) e^{-\frac{1}{2} (1-s^2) a^2}ds\right] \\
&&+ \mathbb{E}\left[a^2 \int_0^1 sa f(\xi + sa) e^{-\frac{1}{2} (1-s^2) a^2} ds \right]\\
&=&  \mathbb{E}\left[a^2 \int_0^1f^\prime(\xi + sa) e^{-\frac{1}{2} (1-s^2) a^2}ds\right]\\
&& + \mathbb{E}\left[a \int_0^1 f(\xi + sa) d\left[e^{\frac{1}{2} s^2 a^2}\right] e^{-\frac{1}{2} a^2}\right]\\
&=& \mathbb{E}\left[a^2 \int_0^1f^\prime(\xi + sa) e^{-\frac{1}{2} (1-s^2) a^2}ds\right]\\
&& +\left(\mathbb{E}\left[a f(\xi + a)\right] - \mathbb{E}\left[a^2 \int_0^1f^\prime(\xi + sa) e^{-\frac{1}{2} (1-s^2) a^2}ds\right]\right)\\
&=& \mathbb{E}\left[a f(\xi + a)\right],
\end{eqnarray*}
where the second equality follows from the fact that, for $h : \mathbb{R} \rightarrow \mathbb{R}$ continuously differentiable, 
\begin{eqnarray}
\mathbb{E}\left[\xi h(\xi)\right] = \mathbb{E}\left[h^\prime(\xi)\right],
\end{eqnarray}
which is seen by performing an integration by parts. Now some standard approximation procedures allow us to drop the condition that $f$ be continuously differentiable.
\end{proof}

\begin{lemA}
Let $\mu$ be a symmetric probability distribution on the real line with Fourier transform
\begin{eqnarray}
\widehat{\mu}(t) = \int_{-\infty}^\infty e^{-itx}d\mu(x).
\end{eqnarray}
Then, for $a \in \mathbb{R}$,
\begin{eqnarray}
\mathbb{E}\left[\widehat{\mu}(\xi + a)\right] = \int_{-\infty}^\infty e^{-\frac{1}{2}x^2} \cos(ax) d\mu(x).\label{FourTransl}
\end{eqnarray}
\end{lemA}

\begin{proof}
This is standard.
\end{proof}

\begin{lemA}
Let $h : \mathbb{R} \rightarrow \mathbb{R}$ be bounded with a piecewise continuous derivative.
Then, for random variables $\eta, \eta^\prime$, 
\begin{eqnarray}
\left|\mathbb{E}\left[h(\eta) - h\left(\eta^\prime\right)\right]\right| \leq \left(\int_{-\infty}^\infty \left|h^\prime (x) \right| dx\right) K\left(\eta,\eta^\prime\right).\label{AbsContEst}
\end{eqnarray} 
\end{lemA}

\begin{proof}
Perform an integration by parts on the left side of (\ref{AbsContEst}).
\end{proof}

\begin{thmA}\label{AbstractThmLowerBound}
Let $\mu$ be a symmetric probability distribution on the real line, put 
\begin{eqnarray}
f(x) = \left\{\begin{array}{clrr}      
\frac{1 - \widehat{\mu}(x)}{x} 	& \textrm{ if }& x \neq 0\\       
0 & \textrm{ if }& x = 0
\end{array}\right.
\end{eqnarray}
and suppose that $f$ has a bounded second derivative and a bounded and piecewise continuous third derivative. Then
\begin{eqnarray}
\lefteqn{\left(\int_{-\infty}^\infty \left[1 - \left(1 + \frac{1}{2}x^2\right) e^{-\frac{1}{2} x^2}\right] d\mu(x)\right) \widetilde{\textrm{\upshape{Lin}}}_{\psi_{1/2}}\left(\left\{\xi_{n,k}\right\}\right)}\label{abstrIneq}\\
&\leq& \left(\int_{-\infty}^\infty \left|\left(\widehat{\mu}\right)^\prime(x)\right| dx + \int_{-\infty}^\infty \left|f^{\prime \prime}(x)\right| dx\right) \limsup_{n \rightarrow \infty} K\left(\xi,\sum_{k=1}^n \xi_{n,k}\right).\nonumber
\end{eqnarray}
\end{thmA}

\begin{proof}
Put 
\begin{eqnarray}
g(x) = xf(x) 
\end{eqnarray}
and 
\begin{eqnarray}
\beta_{n,k} &=& \int_0^1 \left[\widehat{\mu}\left(\xi\right) - \widehat{\mu}\left(\xi + s\xi_{n,k}\right)\right] e^{-\frac{1}{2} (1-s^2) \xi_{n,k}^2} ds,\\
\gamma_{n} &=& g(\xi) - g\left(\sum_{k=1}^n \xi_{n,k}\right),\\
\delta_{n,k} &=& \int_0^1 \left[f^\prime\left(s\xi_{n,k} + \sum_{i \neq k} \xi_{n,i} + \widetilde{\xi}_{n,k}\right) - f^\prime\left(s \xi_{n,k} + \xi\right)\right] ds,\\
\epsilon_{n,k} &=& \int_{0}^1 \left[f^\prime\left(\sum_{i \neq k} \xi_{n,i} + s \xi_{n,k}\right) - f^\prime\left(\sum_{i \neq k} \xi_{n,i} + s\xi_{n,k} + \widetilde{\xi}_{n,k}\right)\right.\\
&& + \left.\widetilde{\xi}_{n,k} f^{\prime \prime}\left(\sum_{i \neq k} \xi_{n,i} + s \xi_{n,k}\right) \right]ds.\nonumber
\end{eqnarray}
Then
\begin{eqnarray}
\sum_{k=1}^n \mathbb{E}\left[\xi_{n,k}^2 \left(1-\widehat{\mu}(\xi)\right)\psi_{1/2}\left(\left|\xi_{n,k}\right|\right)\right]  = \sum_{k=1}^n \mathbb{E}\left[\xi_{n,k}^2 \left(\beta_{n,k} + \gamma_n + \delta_{n,k} + \epsilon_{n,k}\right)\right]
\end{eqnarray}
which is seen by calculating the right side applying independence, the fact that $\mathbb{E}\left[\widetilde{\xi}_{n,k}\right] = \mathbb{E}\left[\xi_{n,k}\right] = 0$ and (\ref{basicRelation}).\\
It follows from (\ref{FourTransl}) and the inequality $\displaystyle{1 - \cos(a) \leq \frac{1}{2} a^2}$ that 
\begin{eqnarray}
\lefteqn{\sum_{k=1}^n \mathbb{E}\left[\xi_{n,k}^2 \beta_{n,k}\right]}\\
&=& \int_{-\infty}^\infty e^{-\frac{1}{2}x^2} \sum_{k=1}^n \mathbb{E} \left[\xi_{n,k}^2 \int_0^1 \left[1 - \cos(s \xi_{n,k} x)\right] e^{-\frac{1}{2} (1-s^2) \xi_{n,k}^2}ds\right] d\mu(x)\nonumber\\
&\leq& \left(\int_{-\infty}^\infty \frac{1}{2} x^2 e^{-\frac{1}{2} x^2} d\mu(x)\right) \sum_{k=1}^n\mathbb{E}\left[\xi_{n,k}^2 \int_0^1 s^2 \xi_{n,k}^2 e^{-\frac{1}{2} (1 - s^2) \xi_{n,k}^2}ds\right]\nonumber\\
&=& \left(\int_{-\infty}^\infty \frac{1}{2}x^2 e^{-\frac{1}{2} x^2} d\mu(x)\right) \sum_{k=1}^n \mathbb{E}\left[\xi_{n,k}^2 \psi_{1/2}\left(\left|\xi_{n,k}\right|\right)\right].\nonumber
\end{eqnarray}
From (\ref{AbsContEst}) we learn that
\begin{eqnarray}
\sum_{k=1}^n \mathbb{E}\left[\xi_{n,k}^2 \gamma_{n}\right] \leq  \left(\int_{-\infty}^\infty \left|\left(\widehat{\mu}\right)^{\prime}(x)\right| dx\right) K\left(\xi,\sum_{k=1}^n \xi_{n,k}\right)
\end{eqnarray}
and
\begin{eqnarray}
\sum_{k=1}^n\mathbb{E}\left[\xi_{n,k}^2 \delta_{n,k}\right] &\leq& \left(\int_{-\infty}^\infty \left|f^{\prime \prime}(x)\right|dx\right) K\left(\xi,\sum_{k=1}^n \xi_{n,k}\right).
\end{eqnarray}
Finally, (\ref{Taylorf}) reveals that
\begin{eqnarray}
\sum_{k=1}^n \mathbb{E}\left[\xi_{n,k}^2 \epsilon_{n,k}\right] \leq \frac{1}{2} \left\|f^{\prime \prime \prime}\right\|_\infty \sum_{k=1}^n \sigma_{n,k}^4.
\end{eqnarray}
Now the lemma follows from (\ref{FelNegCond}).
\end{proof}

\begin{gevA}
There exists a constant $C_{\psi_{1/2}}$, not depending on $\left\{\xi_{n,k}\right\}$, such that 
\begin{eqnarray}
\widetilde{\textrm{\upshape{Lin}}}_{\psi_{1/2}}\left(\left\{\xi_{n,k}\right\}\right) \leq C_{\psi_{1/2}} \limsup_{n \rightarrow \infty} K\left(\xi,\sum_{k=1}^n \xi_{n,k}\right).\label{psiIneq}
\end{eqnarray}
Moreover, (\ref{psiIneq}) can be shown to hold with $\displaystyle{C_{\psi_{1/2}} \leq 20.2}$.
\end{gevA}

\begin{proof}
Let $\mu_\sigma$ be normal with mean 0 and variance $\sigma^2$ and put
\begin{eqnarray}
f_\sigma(x) = \frac{1 - \widehat{\mu}_\sigma(x)}{x}.
\end{eqnarray}
Then
\begin{eqnarray}
\int_{-\infty}^\infty \left[1 - \left(1 + \frac{1}{2}x^2\right) e^{-\frac{1}{2} x^2}\right] d\mu_\sigma(x)
= 1 - \left(1 + \frac{1}{2}\frac{\sigma^2}{1+\sigma^2}\right)\frac{1}{\sqrt{1+\sigma^2}}
\end{eqnarray}
and
\begin{eqnarray}
\int_{-\infty}^\infty \left|\left(\widehat{\mu}_{\sigma}\right)^\prime(x)\right| dx =2.
\end{eqnarray}
Furthermore,
\begin{eqnarray}
\int_{-\infty}^\infty \left|f_\sigma^{\prime \prime}(x)\right|dx = \sigma^2 - 4 f_\sigma^\prime\left(R_\sigma\right)
\end{eqnarray}
with $R_\sigma$ the strictly positive zero of $f_\sigma^{\prime \prime}$.
 
If $\sigma = 1.7$, then $1.4912 < R_\sigma < 1.4914$ and applying (\ref{abstrIneq}) in this case reveals that we can take $C_{\psi_{1/2}} \approx 20.19$.
\end{proof}

\paragraph{\bf{Remark}:} The authors have attempted to sharpen the upper bound for $C_{\psi_{1/2}}$ by applying Theorem \ref{AbstractThmLowerBound} to non-normal probability distributions, but have not succeeded so far.

\begin{proof}(Theorem \ref{thmBHallIneqour})
Combine (\ref{CompPsi}) and (\ref{psiIneq}).
\end{proof}

\section*{Appendix B : The theoretical framework behind the calculations}

\paragraph{} Since there are many metrics which metrize the weak topology, the reader may well ask why we have singled out the Kolmogorov metric $K$ in our estimations and why we would be interested in assessing $\displaystyle{\limsup_{n \rightarrow \infty} K\left(\xi,\sum_{k=1}^n \xi_{n,k}\right)}$ the way we did. The deeper fundamental reason for this is to be found in approach theory (\cite{BeTop},\cite{BeAn},\cite{Lo}). 

\paragraph{} Whereas convergence is handled in the weak topology, notions of approximate convergence can only be handled in a structure that allows for numbers, such as e.g.~a metric. In the theory of approximation in metric or normed spaces such notions exist. Indeed, asymptotic radius and asymptotic center of a sequence  are well-known (\cite{beny}, \cite{bos}, \cite{Edel},\cite{Ha},\cite{lam}). Approach theory makes these concepts available in a much wider setting than merely metric spaces. 

\paragraph{} An approach structure on a set can be thought of as a generalized metric structure in the sense that it is also given by a distance (which we usually denote by $\delta$), but not between pairs of points but between points and sets. As in the case of metric spaces, the natural mappings between approach spaces are contractions. Another analogy with metrics is that a distance $\delta$ has a canonical underlying topology which we sometimes refer to as the topology distancized by $\delta$. Since we do not want to go into details and technicalities here, we refer the interested reader to \cite{Lo}. 

\paragraph{} Why do we then nevertheless end up with a metric, and why the Kolmogorov metric? This is not an arbitrary and ad-hoc choice, it is dictated by the theory, just like weak convergence is dictated by the weak topology. Consider the following setup.

\paragraph{} Let $\mathcal{F}$ (respectively $\mathcal{F}_c$) stand for the collection of probability distributions (respectively continuous probability distributions) on the real line and convolution is denoted by $\star$. Now one of the main principles of Bergstr{\"o}m's direct convolution method (\cite{BergCLT},\cite{Ra}) is that weak convergence in $\mathcal{F}$ of a sequence $(\eta_n)_n$ to $\eta$ is equivalent with uniform convergence of the sequence $\left(\eta_n \star \zeta\right)_n$ to $\eta \star \zeta$ for every continuous $\zeta$. In other words, if we let $\mathcal{T}_w$ stand for the topology of weak convergence on $\mathcal{F}$ and $\mathcal{T}_K$ for the topology of uniform convergence (i.e.~generated by the Kolmogorov metric) on $\mathcal{F}_c$, then $\mathcal{T}_w$ is the weakest topology on $\mathcal{F}$ making all mappings
\begin{eqnarray}
\left(\mathcal{F} \rightarrow \left(\mathcal{F}_c,\mathcal{T}_K\right) : \eta \mapsto \eta \star \zeta\right)_{\zeta \in \mathcal{F}_c}\label{GenSource}
\end{eqnarray}
continuous. Fortunately, we can take the weakest topology with the above property, but this is something we cannot do with metrics, and this is where approach theory comes to the rescue. 

\paragraph{} If we replace the uniform topology $\mathcal{T}_K$ in (\ref{GenSource}) by its generating metric $K$, then we end up with the mappings
\begin{eqnarray}
\left(\mathcal{F} \rightarrow \left(\mathcal{F}_c,K\right) : \eta \mapsto \eta \star \zeta\right)_{\zeta \in \mathcal{F}_c}.\label{NumGenSource}
\end{eqnarray}
We are not able to construct a weakest metric on $\mathcal{F}$, metrizing the weak topology and making all mappings in (\ref{NumGenSource}) contractive, it simply does not exist. However, approach spaces allow for so much flexibility that we are able to construct a weakest distance on $\mathcal{F}$, distancizing the weak topology and making all mappings in (\ref{NumGenSource}) contractive. We are interested in the approach structure determined by the latter distance. We call it the continuity approach structure and we denote it by $\delta_c$. So we introduced $\delta_c$ in exactly the same way as $\mathcal{T}_w$.

\paragraph{} It can be shown that the actual distance $\delta_c$ is given by
\begin{eqnarray}
\delta_c\left(\eta,\mathcal{D}\right)
= \sup_{\mathcal{F}_0} \myinf_{\psi \in \mathcal{D}} \sup_{\zeta \in \mathcal{F}_0} K\left(\eta \star \zeta, \psi \star \zeta \right)
\end{eqnarray}
where $\eta \in \mathcal{F}$ and $\mathcal{D} \subset \mathcal{F}$ and where the first supremum runs over all finite subsets $\mathcal{F}_0$ of $\mathcal{F}_c$.

\paragraph{} The nature of approach structures is such that they allow for many topologi-cal-like considerations. Thus in a topological space we can speak of convergent sequences and in an approach space we can speak of the limit operator of a sequence which in our case  is given by
\begin{eqnarray}
\lambda_c(\eta_n \rightarrow \eta)
= \sup_{\zeta \in \mathcal{F}_c} \limsup_{n \rightarrow \infty} K\left(\eta \star \zeta, \eta_n \star \zeta \right).
\end{eqnarray}
Notice that $(\eta_n)_n$ converges weakly to $\eta$ if and only if $\displaystyle{\lambda_c\left(\eta_n \rightarrow \eta\right) = 0}$.

\paragraph{} The above $\delta_c$ is not a metric, it is a genuine approach structure and $\lambda_c$ is not an asymptotic radius. So why then do we nevertheless end up with just the number $\displaystyle{\limsup_{n \rightarrow \infty} K\left(\xi,\sum_{k=1}^n \xi_{n,k}\right)}$? This is because of the peculiarity of the present set-up. 

\paragraph{} We can prove the following results, where $j(\eta)$ stands for the supremum of the discontinuity jumps of the distribution $\eta$ and where $\delta_K$ stands for the distance generated by the metric $K$ in the usual way, i.e.
\begin{eqnarray}
\delta_K(\eta, \mathcal{D}) = \inf_{\psi \in \mathcal{D}} K(\eta, \psi).
\end{eqnarray}

\begin{thmB}\label{CompDcKThm}
For $\eta \in \mathcal{F}$ and $\mathcal{D} \subset \mathcal{F}$
\begin{eqnarray}
\delta_c(\eta,\mathcal{D}) \leq \delta_{K}(\eta,\mathcal{D}) \leq \delta_c(\eta,\mathcal{D}) + j(\eta).
\end{eqnarray}
\end{thmB}

\begin{gevB}\label{EqDcKc}
For $\eta  \in \mathcal{F}_c$ and $\mathcal{D} \subset \mathcal{F}$
\begin{eqnarray}
\delta_c(\eta,\mathcal{D}) = \delta_{K}(\eta,\mathcal{D}). 
\end{eqnarray}
\end{gevB}

\begin{gevB}\label{EqLimcKc}
For $\eta \in \mathcal{F}_c$ and $\left(\eta_n\right)_n$ in $\mathcal{F}$
\begin{eqnarray}
\lambda_c\left(\eta_n \rightarrow \eta\right) = \limsup_{n \rightarrow \infty} K\left(\eta,\eta_n\right).\label{explLimOp}
\end{eqnarray}
\end{gevB}

\paragraph{} The number $\displaystyle{\limsup_{n \rightarrow \infty} K\left(\xi,\sum_{k=1}^n \xi_{n,k}\right)}$ is therefore not ad-hoc, it is nothing else but the limit-operator of the sequence $\displaystyle{\left(\sum_{k=1}^n \xi_{n,k}\right)_n}$ evaluated in $\xi$ for a canonical approach structure on the set of all probability distributions on the real line, and only in our particular case, where we are dealing with a continuous limit distribution, there is equality between the limit-operator for the continuity approach structure $\delta_c$ and the asymptotic radius for the Kolmogorov metric $K$.

\end{document}